\documentclass[a4paper,reqno]{amsart}
\usepackage{soul} 
\usepackage[all]{xy}
\usepackage{verbatim}
\usepackage{amstext}
\usepackage{amsmath}
\usepackage{amscd}
\usepackage{amsthm}
\usepackage{amsfonts}
\usepackage{enumerate}
\usepackage{graphicx}
\usepackage{latexsym}
\usepackage{mathrsfs}
\usepackage{mathtools}
\xyoption{all}

\usepackage{pstricks}
\usepackage{lscape}
\usepackage{comment}

\newtheorem{theorem}{Theorem}[section]

\newtheorem{lemma}[theorem]{Lemma}
\newtheorem{proposition}[theorem]{Proposition}
\newtheorem{definition-proposition}[theorem]{Definition-Proposition}

\theoremstyle{definition}
\newtheorem{definition}[theorem]{Definition}
\newtheorem{assumption}[theorem]{Assumption}
\newtheorem{remark}[theorem]{Remark}
\newtheorem{example}[theorem]{Example}
\newtheorem{observation}[theorem]{Observation}

\newcommand{\CC}{\mathscr{C}}

\newcommand{\DDD}{\mathsf{D}}
\newcommand{\KKK}{\mathsf{K}}
\renewcommand{\O}{{\mathcal O}}

\renewcommand{\L}{\mathbb{L}}
\renewcommand{\P}{\mathbb{P}}
\newcommand{\Z}{\mathbb{Z}}

\newcommand{\X}{\mathbb{X}}
\renewcommand{\L}{\mathbb{L}}
\newcommand{\bo}{\operatorname{b}\nolimits}

\newcommand{\Ext}{\operatorname{Ext}\nolimits}

\newcommand{\Hom}{\operatorname{Hom}\nolimits}
\newcommand{\End}{\operatorname{End}\nolimits}

\newcommand{\gldim}{\operatorname{gl}.\operatorname{dim}\nolimits}

\newcommand{\RHom}{\mathbf{R}\strut\kern-.2em\operatorname{Hom}\nolimits}

\newcommand{\gen}[1]{\langle #1 \rangle}
\renewcommand{\c}{\vec c}
\newcommand{\y}{\vec y}
\newcommand{\x}{\vec x}

\DeclareMathOperator{\moduleCategory}{\mathsf{mod}} \renewcommand{\mod}{\moduleCategory}
\DeclareMathOperator{\Mod}{\mathsf{Mod}}
\DeclareMathOperator{\proj}{\mathsf{proj}}

\DeclareMathOperator{\thick}{\mathsf{thick}}
\DeclareMathOperator{\coh}{\mathsf{coh}}

\DeclareMathOperator{\qgr}{\mathsf{qgr}}

\DeclareMathOperator{\add}{\mathsf{add}}

\newcommand{\cut}{\ar@{-}@[|(5)]}

\begin{document}
\title{Tilting bundles on orders on $\P^d$}

\author[Iyama]{Osamu Iyama}
\address{O. Iyama: Graduate School of Mathematics, Nagoya University, Chikusa-ku, Nagoya, 464-8602 Japan}
\email{iyama@math.nagoya-u.ac.jp}
\urladdr{http://www.math.nagoya-u.ac.jp/~iyama/}
\thanks{The first author was partially supported by JSPS Grant-in-Aid for Scientific Research 24340004, 23540045, 20244001 and 22224001.}

\author[Lerner]{Boris Lerner}
\address{B. Lerner: Graduate School of Mathematics, Nagoya University, Chikusa-ku, Nagoya, 464-8602 Japan}
\email{blerner@gmail.com}
\urladdr{http://www.math.nagoya-u.ac.jp/~lerner/}
\thanks{The second author was supported by JSPS postdoctoral fellowship program.}


\begin{abstract}
We introduce a class of orders on $\P^d$ called Geigle-Lenzing orders and show that they have tilting bundles.
Moreover we show that their module categories are equivalent to the categories of coherent sheaves on Geigle-Lenzing projective spaces introduced in \cite{HIMO}.
\end{abstract}

\maketitle

\section{Introduction}

Throughout we work over a field $k$.
Moreover, for an order $\Lambda$ we
denote by $\mod \Lambda$ the category of coherent left $\Lambda$-modules.

Weighted projective lines were first introduced by Geigle and Lenzing \cite{GL}
and play an important role in representation theory (e.g. \cite{Me, CK, KLM})
and homological mirror symmetry (e.g. \cite{KST, U}).
It has been pointed out in both \cite{CI} and \cite{RVdB} that the category of coherent sheaves on a weighted
projective line is equivalent to the module category of a hereditary order on $\P^1$, where
by an \emph{order} we mean a certain coherent sheaf of non commutative algebras.
However, until now, this has remained only an observation and has not been
capitalised upon. 
In this paper, we aim to show that the language of orders gives
a quite effective tool to study weighted projective lines and their
generalizations.

Recently in \cite{HIMO}, \emph{Geigle-Lenzing} (\emph{GL}) \emph{projective spaces}
were introduced as a higher dimensional generalization of Geigle-Lenzing
weighted projective lines, and their representation theory was studied.
In this paper we will introduce a certain class of orders on $\P^d$
which we call \emph{Geigle-Lenzing} (\emph{GL}) \emph{orders} on $\P^d$
and prove that they actually give the category of coherent sheaves on GL projective spaces:
\begin{theorem}[Theorem \ref{graded vs order}]
	Let $\X$ be a GL projective space and $\Lambda$ be a GL order of the same type. 
	There exists an equivalence 
	\[\coh \X\simeq \mod \Lambda.\]
\end{theorem}

After Beilinson's work \cite{Be}, various projective varieties are known to be derived equivalent to non-commutative algebras: 
for example, Hirzebruch surfaces \cite{Ki}, rational surfaces \cite{HP}, homogeneous spaces \cite{Kap,Kan,BLV} and so on.
The notion of tilting bundles is crucial to construct derived equivalences.
In representation theory, tilting bundles on Geigle-Lenzing weighted projective lines \cite{GL} play an important role since they give Ringel's canonical algebras \cite{R} as their endomorphism algebras.
One of the basic results in \cite{HIMO} (see also \cite{Ba,IU}) is the existence of tilting bundles on GL projective spaces.
Recall that $T\in\mod \Lambda$ is a \emph{tilting $\Lambda$-module} if it satisfies the 
following two conditions:
\begin{itemize}
	\item Rigidity condition: $\Ext^i_\Lambda(T,T)=0$ for all $i>0$,
	\item Generation condition: $\DDD^{\bo}(\mod A)=\thick T$, where $\thick T$ is the smallest triangulated subcategory of $\DDD^{\bo}(\mod A)$ which is closed under direct summands and contains $T$.
\end{itemize}
The existence of such a tilting bundle gives rise to a derived equivalence between $\Lambda$ and $\End_\Lambda(T)$.

We will give a simple proof of the following result in the language of orders:
\begin{theorem}[Theorem \ref{tilting}]
Let $\Lambda$ be a GL order on $\P^d$. 
\begin{itemize}
\item[(a)] There exists a tilting bundle $T$ in $\mod \Lambda$.
\item[(b)] We have a triangle equivalence $\DDD^{\bo}(\mod \Lambda)\simeq\DDD^{\bo}(\mod\End_\Lambda(T))$.
\item[(c)] $\Lambda$ has global dimension $d$.
\end{itemize}
\end{theorem}
In fact, we will explicitly construct a tilting bundle $T$.
Crucially, our proof is geometric for it uses
the theorem of Beilinson \cite{Be} regarding the existence of a tilting bundle on $\P^d$.

\medskip
A similar construction of tilting bundles in a more general setup will be discussed in a joint work \cite{LO} with Oppermann.

\medskip\noindent
{\bf Acknowledgements.}
The authors thank Kenneth Chan for valuable discussion leading to this work.
They thank Colin Ingalls, Gustavo Jasso and Steffen Oppermann for stimulating discussions.

\section{Tilting bundles on orders}\label{Tilting}

Let $\P^d$ be a projective $d$-space and fix $n\geq 0$ hyperplanes 
${\bf L}=(L_1,\dots,L_n)$ on $\P^d$ as well as weights ${\bf p}=(p_1,\dots,p_n)$ with $p_i\in\Z_{\ge0}$. We assume
that the hyperplanes are in general position in the following sense:
\begin{assumption}\label{general position}
For any subset $\{i_1,\ldots,i_m\}$, the intersection $\bigcap_{k=1}^mL_{i_k}$
is codimension $m$ in $\P^d$ or is empty if $m>d$.
\end{assumption}

For a triple $(\O,I,n)$ of a sheaf of rings $\O$ (or a ring), an ideal sheaf $I$ of $\O$ (or an
ideal) and a
positive integer $n$, let ${\rm T}_n(\O,I)$ be the subsheaf
\[ {\rm T}_n(\O,I)=\left[\begin{array}{ccccc}
		\O&I&\cdots&I&I\\
		\O&\O&\cdots&I&I\\
		\vdots&\vdots&\ddots&\vdots&\vdots\\
		\O&\O&\cdots&\O&I\\
		\O&\O&\cdots&\O&\O
\end{array}\right]\]
of the sheaf ${\rm M}_n(\O)$ of full matrix rings.

For the structure sheaf $\O:=\O_{\P^d}$ of $\P^d$, let
\begin{eqnarray*}
\Lambda_i&:=&{\rm T}_{p_i}(\O,\O(-L_i))\\
\Lambda=\Lambda({\bf L}, {\bf p})&:=& \Lambda_1\otimes_{\O}\cdots\otimes_{\O}\Lambda_n
\end{eqnarray*}
which can be regarded as a suborder of ${\rm M}_{p_1\cdots p_n}(\O_{\P^d})$.
We call $\Lambda$ a \emph{Geigle-Lenzing} (\emph{GL}) \emph{order} on $\P^d$ of \emph{type} $({\bf L},{\bf p})$.
Note that the authors of \cite{CI} call the transpose of 
this $\Lambda$ the \emph{canonical matrix form} of its Morita equivalence class.

\begin{theorem}\label{tilting}
Let $\Lambda$ be a GL order on $\P^d$.
\begin{itemize}
\item[(a)] There exists a tilting bundle $T$ in $\mod \Lambda$ given in \eqref{define T} below.
\item[(b)] We have a triangle equivalence $\DDD^{\bo}(\mod \Lambda)\simeq\DDD^{\bo}(\mod\End_\Lambda(T))$.
\item[(c)] $\Lambda$ has global dimension $d$.
\end{itemize}
\end{theorem}
The construction of $T$ is as follows: First we define a $\Lambda$-module $P$ by
\[P_i:=\left[\begin{array}{c}
\O\\ \O\\ \vdots\\ \O\\ \O
\end{array}\right]\in\mod \Lambda_i\ \mbox{ and }\ 
P:=P_1\otimes_{\O}\cdots\otimes_{\O}P_n\in\mod \Lambda.\]
This is a direct summand of the $\Lambda$-module $\Lambda$ and can be described as
\[P=\Lambda e,\ \mbox{ where }\ e:=e_1\otimes\cdots\otimes e_n\in H^0(\P^d, \Lambda)
\ \mbox{ for }\ 
e_i:=\left[\begin{array}{ccccc}
1&0&\cdots&0&0\\
0&0&\cdots&0&0\\
\vdots&\vdots&\ddots&\vdots&\vdots\\
0&0&\cdots&0&0\\
0&0&\cdots&0&0
\end{array}\right]\in H^0(\P^d, \Lambda_i).\]
Next for each $i=1,\ldots,n$, we define an invertible $\Lambda_i$-bimodule $J_i$ by
\[J_i:=\left[\begin{array}{ccccccc}
\O&\O&\O(-L_i)&\cdots&\O(-L_i)&\O(-L_i)&\O(-L_i)\\
\O&\O&\O&\cdots&\O(-L_i)&\O(-L_i)&\O(-L_i)\\
\O&\O&\O&\cdots&\O(-L_i)&\O(-L_i)&\O(-L_i)\\
\vdots&\vdots&\vdots&\ddots&\vdots&\vdots&\vdots\\
\O&\O&\O&\cdots&\O&\O&\O(-L_i)\\
\O&\O&\O&\cdots&\O&\O&\O\\
\O(L_i)&\O&\O&\cdots&\O&\O&\O
\end{array}\right]\]

The following can be easily checked:

\begin{observation}\label{obs:what J does}
$J_i^\ell\otimes_\Lambda P_i$ is the $(1-\ell)$-th (modulo $p_i$) column of $\Lambda_i\otimes_\O\O(\lceil
\ell/p_i \rceil L_i)$,
where $\lceil x\rceil$ is the smallest integer $a$ satisfying $a\geq x$.
\end{observation}

We define an autofunctor $(-)(\x_i)$, which we abbreviate simply by $(\x_i)$, of $\mod \Lambda_i$ by $(\x_i):=J_i\otimes_\O-\colon\mod \Lambda_i\to\mod \Lambda_i$.
Then we extend this action to $\mod \Lambda$ by first introducing an invertible $\Lambda$-bimodule $I_i$ by
\[I_i:=\Lambda_1\otimes_{\O}\cdots\otimes_{\O}J_i\otimes_{\O}\cdots\otimes_{\O}\Lambda_n\]
and defining an autofunctor $(\x_i)$ of $\mod \Lambda$ for each $i=1,\ldots,n$ by
\[(\x_i):=I_i\otimes_\Lambda-\colon\mod \Lambda\to\mod \Lambda.\]
By a simple matrix multiplication, one can easily check the following:

\begin{observation}
Since $\overbrace{J_i\otimes_{\Lambda_i}\cdots\otimes_{\Lambda_i}J_i}^{p_i}=\Lambda_i\otimes_\O\O(L_i)$, 
we have $\overbrace{I_i\otimes_\Lambda\cdots\otimes_\Lambda I_i}^{p_i}=\Lambda\otimes_\O\O(L_i)$ and

\begin{equation}\label{pixi}
	(p_i\x_i)=-\otimes_\O\O(L_i)\simeq -\otimes_\O\O(1).
\end{equation}

\end{observation}

Next we introduce the following rank $1$ group:
\[\L=\L({\bf p}) = \gen{\x_1,\cdots,\x_n,\c}/({p_i\x_i-\c}\;|\;
1\leq i\leq n)\]
By \eqref{pixi}, we get an objectwise
action of $\L$ on $\mod \Lambda$.

Now we denote by $\L_+$ the submonoid of $\L$ generated by $\x_1,\ldots,\x_n$, and we regard $\L$ as a partially ordered set by: $\x\le \y$ if and only if $\y-\x\in\L_+$.
If we let $[0,d\c]:=\{\x\in\L\mid0\le\x\le d\c\}$
then the $\Lambda$-module
\begin{equation}\label{define T}
T:=\bigoplus_{\x\in[0,d\c]}P(\x)
\end{equation}
gives a tilting bundle in Theorem \ref{tilting}. 

\medskip
A proof of Theorem \ref{tilting} is given in the rest of this section.
The basic idea is to reduce 
our problem for $\mod \Lambda$ to the corresponding one in $\mod\O$ and use the
following Beilinson's result:

\begin{theorem}\label{beilinson}\cite{Be}
$\bigoplus_{i=0}^d\O(i)$ is a tilting bundle in $\coh\P^d$.
\end{theorem}

In particular we have
\begin{itemize}
\item $H^i(\P^d,\O(\ell))=0$ for all $i>0$ and all $\ell$ with $-d\le \ell\le d$.
\item $\DDD^{\bo}(\coh\P^d)=\thick\bigoplus_{i=0}^d\O(i)$.
\end{itemize}
The second condition implies that, if $X\in\coh\P^d$ satisfies $H^i(\P^d, X(-\ell))=0$ for all $i\ge0$ and all $\ell$ with $0\leq \ell\leq d$, then $X=0$.

\subsection{Proof the rigidity condition}

In this section we prove the following:

\begin{proposition}\label{ext vanishing}
We have $\Ext^i_\Lambda(T,T)=0$ for any $i>0$.
\end{proposition}

First we need the following observation:

\begin{lemma}\label{cohomology}
For any idempotent $e$ of $H^0(\P^d, \Lambda)$, we have $\Ext^i_\Lambda(\Lambda e,X)\simeq H^i(\P^d,eX)$ for all $X\in\mod \Lambda$ and $i\ge0$.
\end{lemma}

\begin{proof}
We only have to show the case $i=0$ since both sides are the 
right derived functors.
The case $i=0$ follows from the following isomorphism of $k$-vector spaces.
\[\Hom_\Lambda(\Lambda e,-)\simeq e\Hom_\Lambda(\Lambda,-)\simeq eH^0(\P^d, -)\simeq H^0(\P^d, e-),\]
where the middle equality follows from a natural isomorphism $\Hom_\Lambda(\Lambda,-)\simeq H^0(\P^d, -)$.
\end{proof}

Next we show the following:

\begin{lemma}\label{first entry 2}
\begin{itemize}
\item[(a)] Let $\ell\in\Z$. Then $e_i(P_i(\ell\x_i))\simeq\O(\lfloor\ell/p_i\rfloor)$ where 
	$\lfloor x\rfloor$ is the largest integer $a$ satisfying $a\leq x$.

\item[(b)] For any $\x,\y\in[0,d\c]$, we have $e(P(\y-\x))\simeq\O(\ell)$ for some $\ell$ with $-d\le \ell\le d$.
\end{itemize}
\end{lemma}

Notice that $e_i(P_i(\x_i))$ can not be written as $(e_iP_i)(\x_i)$ since $(\x_i)$ is defined for $\Lambda$-modules, not for $\O$-modules.

\begin{proof}
	(a) Follows from Observation \ref{obs:what J does} and \eqref{pixi}.


(b) By our definition of the group $\L$, we can write
\[\y-\x=\sum_{i=1}^n\ell_i\x_i+\ell\c\]
for $0\le\ell_i<p_i$ and $-d\le \ell\le d$.
By (a), we have $e_i(P_i(\ell_i\x_i))=\O$. 
Furthermore, using (\ref{pixi}), we see that $e_i(P_i(\ell\c))=\O(\ell)$. Thus
\[e(P(\y-\x))=(e_1(P_1(\ell_1\x_1))\otimes_{\O}\cdots\otimes_{\O}e_n(P_n(\ell_n\x_n))(\ell)=\O(\ell).\qedhere\]
\end{proof}

Now we are ready to prove Proposition \ref{ext vanishing}.

Let $\x,\y\in[0,d\c]$.
Since $(\x)$ is an autofunctor of $\mod \Lambda$, we have
\[\Ext^i_\Lambda(P(\x),P(\y))=\Ext^i_\Lambda(\Lambda e,P(\y-\x)).\]
By Lemmas \ref{cohomology} and \ref{first entry 2}, we have
\[\Ext^i_\Lambda(\Lambda e,P(\y-\x))=H^i(\P^d, e(P(\y-\x)))=H^i(\P^d, \O(\ell))\]
for some $\ell$ with $-d\le \ell\le d$. This is zero by Theorem \ref{beilinson} and so the proof is completed.
\qed

\subsection{Proof of the generation condition}

The proof is broken up into two parts: first we will prove a seemingly weaker generation condition,
and then show that in our case it is in fact sufficient.

\begin{proposition}\label{generate}
If $X\in\mod \Lambda$ satisfies $\Ext^i_\Lambda(T,X)=0$ for all $i\ge0$, then $X=0$.
\end{proposition}

\begin{proof}
We call an idempotent $e'_i\in H^0(\P^d, \Lambda_i)$ \emph{standard} if it has
one entry $1$ on the diagonal and all other entries are $0$.

Assume $X\in\mod \Lambda$ satisfies $\Ext^i_\Lambda(T,X)=0$ for all $i\ge0$.
We need to show that $e'X=0$ for all idempotents $e'\in H^0(\P^d, \Lambda)$ of the form
\[e'=e'_1\otimes\cdots\otimes e'_n\]
for standard idempotents $e'_i\in H^0(\P^d, \Lambda_i)$.
We prove $e'X=0$ by showing $H^i(\P^d, e'X(-\ell))=0$ for sufficiently many $\ell$ (depending on the support of $e'X$) and then invoke Theorem \ref{beilinson}. We proceed 
by using the induction with respect to $N:=|\{i\mid 1\le i\le n,\ e'_i\neq e_i\}|$.
First we show when case $N=0$ (i.e. $e'=e$):

\begin{lemma}\label{eX=0}
We have $eX=0$.
\end{lemma}

\begin{proof}
For each $\ell$ with $0\le \ell\le d$, we have $\ell\c\in[0,d\c]$.
By Lemma \ref{cohomology} and our assumption we have
\begin{equation}
H^i(\P^d, eX(-\ell))=\Ext^i_\Lambda(P(\ell\c),X)\subset\Ext^i_\Lambda(T,X)=0
\end{equation}
for all $i\ge0$. By Theorem \ref{beilinson}, we have $eX=0$.
\end{proof}

The case $N=1$ follows from the following lemma:
\begin{lemma}\label{killebyideal}
Assume that $e_i\neq e'_i$ and $e_j=e'_j$ for any $j\neq i$.
\begin{itemize}
\item[(a)] The $\O$-module $e'X$ is annihilated by $\O(-L_i)$.
\item[(b)] We have $e'X=0$.
\end{itemize}
\end{lemma}

\begin{proof}
(a) Under the natural identification $e'\Lambda e'=\O$, we have $e'\Lambda e\Lambda e'=\O(-L_i)$.
We thus have by Lemma \ref{eX=0}
\[\O(-L_i)e'X=(e'\Lambda e\Lambda e')(e'X)=(e'\Lambda e)(\Lambda e'X)\subset e'\Lambda eX=0.\]

(b) By (a), we can regard $e'X$ as a sheaf on $L_i\cong\P^{d-1}$.

For each $\ell$ with $0\le \ell<d$, and $0<m<p_1$ we have 
$m\x_1+\ell\c\in[0,d\c]$.
Since $\Lambda e'=P(m\x_1)$, it follows from Lemma \ref{cohomology} that
\begin{equation}
	H^i(\P^d,e'X(-\ell))=\Ext^i_\Lambda(\Lambda e'(\ell\c),X)=\Ext^i_\Lambda(P(m\x_1+\ell\c),X)\subset\Ext^i_\Lambda(T,X)=0
\end{equation}
for all $i\geq0$.
Thus we have
\[H^i(\P^{d-1},e'X(-\ell))=H^i(\P^d,e'X(-\ell))=0\]
for all $\ell$ with $0\le \ell<d$. By Theorem \ref{beilinson} (replace $d$ there by $d-1$), we have $e'X=0$.
\end{proof}

The case $N\ge2$ can be shown similarly:

By Assumption \ref{general position} and a similar argument as in the proof of Lemma \ref{killebyideal}(a),
if $d-N\geq 0$ we can regard $e'X$ as a
sheaf on $\bigcap_{e'_i\neq e_i}L_i \cong \P^{d-N}$ and if $d-N<0$ it follows that $e'X=0$.

On the other hand, by $\Ext^i_\Lambda(T,X)=0$ for all $i\ge0$, we have
\[H^i(\P^d, e'X(-k))=H^i(\P^{d-N}, e'X(-k))=0\]
for all $i\ge0$ and all $k$ with $0\le k\le d-N$.
By Theorem \ref{beilinson} (replace $d$ there by $d-N$) we have $e'X=0$.
\end{proof}

We will now show that Proposition \ref{generate} implies the generation condition.
To do this, we first show that $\Lambda$ has global dimension $d$.
\begin{lemma}\label{gldimmat}
Let $(R,{\mathfrak m})$ be a regular local ring of dimension $d$ and
$(a_1,\ldots,a_\ell)$ a subset of a minimal set of generators of $\mathfrak{m}$. Then the ring
\[B:={\rm T}_{p_1}(R,(a_1))\otimes_R\cdots\otimes_R{\rm T}_{p_\ell}(R,(a_\ell))\]
has global dimension $d$.
\end{lemma}

\begin{proof}
It is enough to show that any simple $B$-module $S$ has projective dimension $d$.
Up to an automorphism of $B$, we can assume that
\[S=\left[\begin{array}{c}
R/{\mathfrak m}\\ 0\\ \vdots\\ 0\\ 0
\end{array}\right]\]
where we regard $B$ as a subring of $M_{p_1\cdots p_\ell}(R)$.

Let
\[M:=\left[\begin{array}{c}
R/(a_1,\ldots,a_\ell)\\ 0\\ \vdots\\ 0\\ 0
\end{array}\right]\simeq
\bigotimes_{i=1}^\ell\left[\begin{array}{c}
R/(a_i)\\ 0\\ \vdots\\ 0\\ 0
\end{array}\right]\in\mod B.\]
Since $R/{\mathfrak m}$ is an $R/(a_1,\ldots,a_\ell)$-module with
projective dimension $d-\ell$, we have an exact sequence
\[0\to M_{d-\ell}\to\cdots\to M_0\to S\to0\]
of $B$-modules with $M_i\in\add M$.
On the other hand, the $B$-module $M$ has projective dimension $\ell$
since it has a projective resolution
\[\bigotimes_{i=1}^\ell\left(
\left[\begin{array}{c}
(a_i)\\ R\\ \vdots\\ R\\ R
\end{array}\right]\to
\left[\begin{array}{c}
R\\ R\\ \vdots\\ R\\ R
\end{array}\right]
\right).\]
Thus the $B$-module $S$ has projective dimension $d$.
\end{proof}

\begin{proposition}\label{global dimension}
\begin{itemize}
\item[(a)] The order $\Lambda_x$ has global dimension $d$ for all closed points $x\in\P^d$.
\item[(b)] The order $\Lambda$ has global dimension $d$.
\end{itemize}
\end{proposition}

\begin{proof}
(a) Note that $\O_x:=\O_{\P^d,x}$ is a regular local ring of dimension $d$.
Let $I_x:=\{i\mid 1\le i\le n,\; x\in L_i\}$. Then we have
\[(\Lambda_i)_x=\left\{\begin{array}{cc}
{\rm T}_{p_i}(\O_x,(a_i))&\mbox{if $i\in I_x$,}\\
{\rm M}_{p_i}(\O_x)&\mbox{if $i\notin I_x$.}
\end{array}\right.\]
for $a_i\in \O_x$ defining $L_i$ locally. Our Assumption \ref{general position} implies that $(a_i)_{i\in I_x}$ is a subset of a minimal set of generators of the maximal ideal
$\mathfrak{m}_x$ of $\O_x$.
Thus $\Lambda_x$ is Morita-equivalent to
\[\bigotimes_{i\in I_x} {\rm T}_{p_i}(\O_x,(a_i)),\]
and the statement now follows from Lemma \ref{gldimmat}.

(b) For $X,Y\in\mod\Lambda$, we need show that $\Ext^i_\Lambda(X,Y)=0$ for all $i>d$.
We use an argument similar to \cite[Theorem 1(4)]{BD}.

We first prove that the support of the $\O_{\P^d}$-module
${\mathscr Ext}^q_\Lambda(X,Y)$ has dimension at most $d-q$.
We need to show that ${\mathscr Ext}^q_\Lambda(X,Y)_x=0$ holds
for any point $x\in\P^d$ with $\dim x>d-q$.
For a closed point $y$ on $x$, we know $\gldim\Lambda_y=d$ by (a).
Since $\Lambda_x=\Lambda_y\otimes_{\O_y}\O_x$, we have
\[\gldim \Lambda_x=\dim\O_x=d-\dim x\]
(e.g. apply \cite[Theorem 2.17(2)$\Rightarrow$(1)]{IW} with
$(R,\Lambda):=(\O_y,\Lambda_y)$ there).
Therefore ${\mathscr Ext}^q_\Lambda(X,Y)_x=\Ext^q_{\Lambda_x}(X_x,Y_x)=0$
holds if $q>d-\dim x$, and the assertion follows.

In particular, if $p+q>d$, then $H^p(\P^d,{\mathscr Ext}^q_\Lambda(X,Y))=0$ holds.
Using the local-global spectral sequence
$H^p(\P^d,{\mathscr Ext}^q_\Lambda(X,Y))\Rightarrow \Ext^{p+q}_\Lambda(X,Y)$
(e.g. \cite[Theorem 1(2)]{BD}), we have
$\Ext^i_\Lambda(X,Y)=0$ for all $i>d$.
\end{proof}

Now we are ready to prove that $T$ generates the category.

\begin{proposition}\label{generate 2}
\begin{itemize}
\item[(a)] $\DDD^{\bo}(\mod \Lambda)=\thick T$.
\item[(b)] We have a triangle equivalence $\DDD^{\bo}(\mod \Lambda)\simeq\DDD^{\bo}(\mod\End_\Lambda(T))$.
\end{itemize}
\end{proposition}
We prepare some notions. Let $\CC$ and $\CC'$ be additive subcategories of $\DDD^{\bo}(\mod \Lambda)$.

We call $\CC$ \emph{contravariantly finite} if for any $X\in\DDD^{\bo}(\mod \Lambda)$, there exists a morphism $f:C\to X$ with $C\in\CC$ such that
\[\Hom_{\DDD^{\bo}(\mod \Lambda)}(C',C)\xrightarrow{f}\Hom_{\DDD^{\bo}(\mod \Lambda)}(C',X)\]
is surjective for any $C'\in\CC$.

We define an additive subcategory $\CC*\CC'$ of $\DDD^{\bo}(\mod \Lambda)$ by
\[\CC*\CC':=\{X\in\DDD^{\bo}(\mod \Lambda)\mid {\rm there\;exists\;a\;triangle\;} C\to X\to C'\to C[1]\mbox{ with }\ C\in\CC,\ C'\in\CC'\}.\]
If $\CC$ and $\CC'$ are contravariantly finite, then so is $\CC*\CC'$ (see \cite[Theorem 1.3]{Ch}).

\begin{proof}[Proof of Proposition \ref{generate 2}]
(a) For any $X\in\DDD^{\bo}(\mod \Lambda)$, there exists only finitely many $i\in\Z$ such that
$\Hom_{\DDD^{\bo}(\mod \Lambda)}(T,X[i])\neq0$ since $\Lambda$ has finite global dimension by Proposition \ref{global dimension}.
Thus $\CC:=\add\{ T[i]\mid i\in\Z\}$ is a contravariantly finite subcategory of $\DDD^{\bo}(\mod \Lambda)$. 
In particular, $\CC*\cdots*\CC$ ($m$ times) is also contravariantly finite for all $m\ge0$.

Let $E:=\End_\Lambda(T)$.
Then $\thick T$ is triangle equivalent to $\KKK^{\bo}(\proj E)$ \cite{Ke}.
Since the quiver of $E$ is clearly acyclic (e.g. Section 4), $E$ has finite
global dimension $m$. Thus we have
\[\thick T=\CC*\cdots*\CC\ \ \ (m+1\ \mbox{times})\]
(see e.g. \cite[Proposition 2.6]{KK}).
In particular, $\thick T$ is contravariantly finite.
Applying \cite[Proposition 2.3(1)]{IY}, for any $X\in\DDD^{\bo}(\mod \Lambda)$,
there exists a triangle
$Y\to X\to Z\to Y[1]$ with $Y\in\thick T$ and $\Hom_{\DDD^{\bo}(\mod \Lambda)}(U,Z)=0$
for any $U\in\thick T$.
By Proposition \ref{generate}, we have $Z=0$ and $X\simeq Y\in\thick T$.
Thus the assertion follows.

(b) We already observed $\DDD^{\bo}(\mod \Lambda)=\thick T\simeq\KKK^{\bo}(\proj E)$.
This is $\DDD^{\bo}(\mod E)$ since $E$ has finite global dimension.
\end{proof}

\section{Explicit correspondence between GL orders on
$\P^d$ and GL weighted $\P^d$}\label{correspondence}


\subsection{A graded Morita equivalence}
In this section we show, given a ring graded by a commutative group,
how to modify the ring, so that it is graded by a subgroup.

Let $G$ be an abelian group and $A$ be a $G$-graded ring.
We denote by $\Mod^GA$ the category of $G$-graded $A$-modules, and by $\mod^GA$ the category of finitely generated $G$-graded $A$-modules.

For a subgroup $H<G$ with finite index, we fix a complete set of representatives $I\subseteq G$ of $G/H$. Let
\begin{equation}\label{eqn:ringchange}
A^{[H]}:=\bigoplus_{h\in H}(A^{[H]})_h\ \ \ \mbox{where}\ \ \ 
(A^{[H]})_h:=\left( A_{i-j+h} \right)_{i,j\in I}.
\end{equation}
Then $A^{[H]}$ has a structure of an $H$-graded ring whose multiplication $(A^{[H]})_h\times (A^{[H]})_{h'}\to (A^{[H]})_{h+h'}$ for $h,h'\in H$ is given by 
\[(a_{i,j})_{i,j\in I}\cdot (a'_{i,j})_{i,j\in I}
:=\left(\sum_{k\in I}a_{i,k}\cdot a'_{k,j}\right)_{i,j\in I}.\]
It is easy to see that the ring structure of $A^{[H]}$ does not depend on the choice of $I$.
Moreover the choice of $I$ does not change the graded structure of $A^{[H]}$ up to graded-Morita equivalence
in the following sense.

\begin{theorem}\label{graded Morita}
With the notation above, we have an equivalence of categories:
\[\Mod^GA\simeq \Mod^HA^{[H]}\]
which induces an equivalence $\mod^GA\simeq \mod^HA^{[H]}$.
 \end{theorem}

Although similar results already exist (e.g. \cite{H,Mo}),
we include a complete proof due to lack of suitable references for our setting.

 \begin{proof}
For $M=\bigoplus_{g\in G}M_g$ in $\Mod^GA$, define $FM\in\Mod^HA^{[H]}$ by
\[FM:=\bigoplus_{h\in H}(FM)_h\ \ \ \mbox{where}\ \ \ (FM)_h:=(M_{i+h})_{i\in I}.\]
Thus $FM=M$ as abelian groups, and the action $(A^{[H]})_h\times (FM)_{h'}\stackrel{\rm act.}\longrightarrow
(FM)_{h+h'}$ is given by
\[(a_{i,j})_{i,j\in I}\cdot (m_{i})_{i\in I}
:=\left(\sum_{k\in I}a_{i,k}\cdot m_{k}\right)_{i\in I}.\]

Let $f:M\to N$ be a morphism of abelian groups. Then $f$ is a morphism in $\Mod^GA$ if and only if $f$ induces the following commutative diagram for any $i,j\in G$:
\[\xymatrix{
A_i\times M_j\ar[r]^(.6){{\rm act.}}\ar[d]^{1\times f}&M_{i+j}\ar[d]^f\\
A_i\times N_j\ar[r]^(.6){{\rm act.}}&N_{i+j}.}\]
This is equivalent to that $f$ induces the following commutative diagram for any $h,h'\in H$:
\[\xymatrix{
(A_{i-j+h})_{i,j\in I}\times(M_{i+h'})_{i\in I}\ar[r]^(.65){{\rm act.}}\ar[d]^{1\times f}&(M_{i+h+h'})_{i\in I}\ar[d]^f\\
(A_{i-j+h})_{i,j\in I}\times(N_{i+h'})_{i\in I}\ar[r]^(.65){{\rm act.}}&(N_{i+h+h'})_{i\in I}.}\]
This means that $f$ is a morphism in $\Mod^HA^{[H]}$.
We conclude that $F:\Mod^GA\to\Mod^HA^{[H]}$ is fully faithful.

Finally we show that $F$ is dense.
For $i\in I$, let $e_i\in (A^{[H]})_0$ be an idempotent whose $(i,i)$-entry is $1$ and other entries are $0$.
For any $N=\bigoplus_{h\in H}N_h$ in $\Mod^HA^{[H]}$, let
\[M_{i+h}:=e_iN_h\ \ \ \mbox{for $i\in I$, $h\in H$ and}\ \ \ M:=\bigoplus_{g\in G}M_g.\]
For $g,g'\in G$, we define the action $A_g\times M_{g'}\to M_{g+g'}$ by
\[A_g\times M_{g'}=e_i(A^{[H]})_he_j\times e_jN_{h'}\xrightarrow{{\rm act.}} e_iN_{h+h'}=M_{g+g'},\]
where $i,j\in I$ and $h,h'\in H$ are unique elements satisfying $g'=j+h'$ and $g=i-j+h$. It is routine to check that $M$ is a $G$-graded $A$-module satisfying $FM\simeq N$.

It remains to show that $F$ induces an equivalence $\mod^GA\to\mod^HA^{[H]}$.
This is immediate since $F(A(i+h))=(A^{[H]}(h))e_i$ holds for any $i\in I$ and $h\in H$.
\end{proof}

\begin{remark}
Theorem \ref{graded Morita} can also be shown by the following argument:
Let $P:=\bigoplus_{i\in I}A(-i)\in\Mod^GA$.
Then the following statements can be checked easily:
\begin{itemize}
\item $\End_A^G(P)=A^{[H]}$.
\item $\add\{P(h)\mid h\in H\}=\proj^GA$.
\end{itemize}
By a standard argument in Morita theory, one can show that the functor
\[F:=\bigoplus_{h\in H}\Hom_A^G(P,-(h)):\Mod^GA\to\Mod^HA^{[H]}\]
is an equivalence.
\end{remark}

 \begin{example}
	 Let $A$ be a $\Z$-graded ring. If we choose $\left\{ 0,1 \right\}$ as the
	 representatives of the two cosets of $\Z/2\Z$ then
\[ A^{[2\Z]}=\bigoplus_{i\in 2\Z}(A^{[2\Z]})_i\ \ \ \mbox{where}\ \ \ (A^{[2\Z]})_{i}=
\begin{bmatrix}A_{i}&{A_{i-1}}\\A_{i+1}&A_{i}\end{bmatrix}.\]
 A $\Z$-graded $A$-module $M=\bigoplus_{i\in\Z}M_i$ corresponds
	 to the $2\Z$-graded $A^{[2\Z]}$-module $
	 \bigoplus_{i\in 2\Z}\begin{bmatrix}
		  M_{i}\\M_{i+1}
	 \end{bmatrix}
	 $.
 \end{example}
\subsection{Geigle-Lenzing projective spaces.}
We now introduce Geigle-Lenzing (GL)
projective spaces, or more precisely, the category of coherent sheaves
on them. The technique of studying a category resembling a category of sheaves
without ever explicitly mentioning a topological space is especially prominent in
noncommutative algebraic geometry.
We use the same notation as in \cite{HIMO} where much more information
can be found. Our goal
is to show that the category of coherent sheaves on GL projective spaces
is equivalent to the module category of a GL order on $\P^d$.

To define GL weighted $\P^d$, as before we choose $n$ hyperplanes ${\bf L}=(L_1,\dots,
L_n)$ in $\P^d_{T_0:\dots : T_d}$
satisfying Assumption \ref{general position}.  We may assume the hyperplanes are given as
zeros of the linear polynomials
\[\ell_i({\bf T})=\sum_{j=0}^{d}\lambda_{i,j}T_j.\]
	Also fix an $n$-tuple of positive integers (the \emph{weights}) ${\bf p}=(p_1,\dots,p_n)$
and let 
\[k[{\bf T},{\bf X}]=k[T_0,\dots,T_d,X_1,\dots,X_n]\] 
\begin{equation}\label{relations}
h_i:=X_i^{p_i}-\ell_i({\bf T)}
\end{equation}
Now consider the $k$-algebra 
\[R=R({\bf L},{\bf p}):=k[{\bf T},{\bf X}]/(h_i\;|\;1\leq i\leq n).\] 
As in the previous section, let \[\L=\L({\bf p}) = \gen{\x_1,\cdots,\x_n,\c}/({p_i\x_i-\c}\;|\;
1\leq i\leq n)\]
We give $R$ an $\L$-grading by defining ${\rm deg}\;X_i=\x_i$ and
${\rm deg}\;T_i=\c$.
	
We will soon encounter several different graded rings so it is useful to establish the following notation:

\begin{definition}
	Let $G$ be an abelian group and $A=\bigoplus_{g\in G}A_g$ be a right noetherian $G$-graded ring which is finitely generated over $k$ and $\dim_kA_g<\infty$ for all $g\in G$. 
	We denote by $\mod^GA$ the category of finitely generated $G$-graded $A$-modules, and by $\mod^G_0A$ the full subcategory of $\mod^GA$ of finite dimensional modules.
We let
\[\qgr A:=\mod^G A/\mod^G_0A.\]
\end{definition}

We apply this definition to the setting of GL projective spaces:

\begin{definition}
For the $\L$-graded $k$-algebra $R=R({\bf L}, {\bf p})$, we call
	\[\coh\X=\coh\X({\bf L}, {\bf p}) :=\qgr R({\bf L}, {\bf p})\]
	the category of \emph{coherent sheaves on GL projective space ${\X}$ of type} $({\bf L},{\bf p})$.
\end{definition}

In the rest of this section, we will prove the following connection between GL projective spaces and GL orders.

\begin{theorem}\label{graded vs order}
Let $\X=\X({\bf L}, {\bf p})$ be a GL projective space and $\Lambda=\Lambda({\bf L}, {\bf p})$ be a GL order of type $({\bf L}, {\bf p})$. Then we have an equivalence
\[\coh\X\simeq \mod \Lambda.\]
\end{theorem} 

For the subgroup $\Z\c$ of $\L$, we have
\[\Z\c\simeq\Z\ \mbox{ and }\ \L/\Z\c\simeq\prod_{i=1}^n\Z/ p_i \Z.\]
A key role in the proof is played by a $\Z$-graded subring
\[S:=\bigoplus_{\ell\in\Z}R_{\ell\c}\]
which is isomorphic to the polynomial ring $k[{\bf T}]$.
By (\ref{relations}), we have $X_i^{p_i}=\ell_i({\bf T})\in S$ for all $1\leq i\leq n$.

Now let us consider the $\Z$-graded ring $R^{[\Z\c]}$.

\begin{proposition}\label{prop:standardfrorm}
We have an isomorphism of $\Z$-graded $k$-algebras
\[R^{[\Z\c]}\simeq {\rm T}_{p_1}(S,X_1^{p_1})\otimes_S\cdots\otimes_S {\rm T}_{p_n}(S,X_n^{p_n}),\]
where we regard ${\rm T}_{p_i}(S,X_i^{p_i}):={\rm T}_{p_i}(S,(X_i^{p_i}))$ as a $\Z$-graded $k$-algebra whose degree $\ell$-part
consists of elements such that lower diagonal entries are in $S_\ell$ and upper diagonal entries are in $X_i^{p_i}S_{\ell-1}$.
\end{proposition}

\begin{proof}
Let $I:=\{\sum_{i=1}^na_i\x_i\mid 0\le a_i\le p_i-1\}$ be a complete set of representatives of $\L/\Z\c$.
Let $S_i:=k[X_i^{p_i}]$ for $1\le i\le n$.
For $\x=\sum_{i=1}^na_i\x_i$ and $\y=\sum_{i=1}^nb_i\x_i$ in $I$, we have
\begin{eqnarray*}
R_{\x-\y+\Z\c}=(\prod_{i=1}^nX_i^{a_i-b_i+\epsilon_ip_i})S=(X_1^{a_1-b_1+\epsilon_1p_1}S)\otimes_S\cdots\otimes_S(X_n^{a_n-b_n+\epsilon_np_n}S),
\end{eqnarray*}
where $\epsilon_i:=0$ if $a_i\ge b_i$ and $1$ otherwise. Thus we have isomorphisms
\begin{eqnarray*}
R^{[\Z\c]}&=&\bigotimes_{i=1}^n\left[
\begin{array}{ccccc}
S&X_i^{p_i-1}S&\cdots&X_i^2S&X_iS\\
X_iS&S&\cdots&X_i^3S&X_i^2S\\
\vdots&\vdots&\ddots&\vdots&\vdots\\
X_i^{p_i-2}S&X_i^{p_i-3}S&\cdots&S&X_i^{p_i-1}S\\
X_i^{p_i-1}S&X_i^{p_i-2}S&\cdots&X_iS&S
\end{array}\right]\\
&=&{\rm T}_{p_1}(S,X_1^{p_1})\otimes_S\cdots\otimes_S {\rm T}_{p_n}(S,X_n^{p_n})
\end{eqnarray*}
of $\Z$-graded $k$-algebras.
\end{proof}


	Let $\Lambda=\Lambda({\bf L},{\bf p})$ and $L=\Lambda\otimes_{\O}\O(1)$ which is a $\Lambda$-bimodule. 
	We define a $\Z$-graded ring
\[B(\Lambda,L):=\bigoplus_{\ell=0}^\infty H^0\left(\P^d, L^{\otimes_\Lambda\ell}\right).\]


\begin{proposition}\label{B and R^Zc}
We have an isomorphism of $\Z$-graded $k$-algebras:
\[B(\Lambda,L)\simeq R^{[\Z\c]}.\]
\end{proposition}

\begin{proof}
	Clearly we have
\[\bigoplus_{\ell=0}^\infty H^0(\P^d, \O(\ell))\simeq 
	S=k[{\bf T}].\] and so we get a category equivalence
	\[\Phi:=\bigoplus_{\ell=0}^\infty H^0(\P^d,-(\ell))
\!:\coh \P^d\simeq \mod^{\Z}S/\mod_0^{\Z}S=\qgr S.\] 
Since the divisor $L_i$ is the zero set of the polynomial $\ell_i(\bf T)$,
the functor $\Phi$ sends the natural inclusion $\O(-L_i)\to\O$ to
the natural inclusion $\ell_i({\bf T}) S\to S$.
	Thus we get 
	\begin{align*}
	B(\Lambda,L)&\simeq
	{\rm T}_{p_1}(S,\ell_1({\bf T}))\otimes_S\cdots\otimes_S {\rm T}_{p_n}(S,\ell_n({\bf T}))\\
	&={\rm T}_{p_1}(S,{X_1^{p_1}})\otimes_S\cdots\otimes_S {\rm T}_{p_n}(S,X_{n}^{p_n})\\
	&=R^{[\Z\c]}.\qedhere
	\end{align*}
\end{proof}

To see the role played by $B(\Lambda,L)$ we first need the following more general set up 
introduced by Artin-Zhang in \cite{AZ}.
Let ${\mathcal C}$ be a $k$-linear abelian category, $P\in{\mathcal C}$  a distinguished object. For any $M\in {\mathcal C}$  we define $H^0(M):=\Hom_{{\mathcal C}}(P,M)$. Assume that:
	\begin{enumerate}
		\item[(H1)] $P$ is a noetherian object,
		\item[(H2)] $A_0:=H^0(P)=\End_{{\mathcal C}}(P)$ is a right noetherian ring and $H^0(M)$ is a finitely generated $A_0$-module for all $M\in{\mathcal C}$.
	\end{enumerate}
			Furthermore, let $s$ be a $k$-linear automorphism of ${\mathcal C}$ satisfying the following assumption:
	\begin{enumerate}
		\item[(H3)] $s$ is \emph{ample} in the following sense:
			\begin{enumerate}
				\item for every $M\in{\mathcal C}$, there are positive integers $\ell_1,\dots,\ell_p$ and an epimorphism $\bigoplus_{i=1}^ps^{-\ell_i}P\to M$ in ${\mathcal C}$,
				\item for every epimorphism $f:M\to N$ in ${\mathcal C}$ there exists an integer  $\ell_0$ 
					such that for every $\ell\geq\ell_0 $ the map
					$H^0(s^\ell f):H^0(s^\ell M)\to H^0(s^\ell N)$ is surjective.
			\end{enumerate}
	\end{enumerate}
	Using this setup, we can construct the following $\Z$-graded ring:
\[B := \bigoplus_{\ell=0}^\infty H^0(s^\ell P)\]
and we have the following crucial result, which can be viewed as a generalization of Serre's theorem:
\begin{theorem}\cite[Theorem 4.5]{AZ}\label{Artin-Zhang}
	$B$ is a right noetherian $k$-algebra, and 
	there is an equivalence of categories
\[{\mathcal C}\simeq \qgr B\]
given by $M\mapsto \bigoplus_{\ell=0}^{\infty}H^0(s^\ell M)$.
\end{theorem}

Now we are ready to prove Theorem \ref{graded vs order}.

Using Theorem \ref{graded Morita} and Propositions \ref{B and R^Zc}, we have equivalences
\begin{eqnarray*}
	\coh\X&=&\qgr R\\
	&\simeq&\qgr R^{[\Z\c]}\\
	&\simeq&\qgr B(\Lambda,L).
\end{eqnarray*}
We specialise Artin-Zhang Theorem \ref{Artin-Zhang} to our case by letting ${\mathcal C}:=\mod \Lambda$, $s:=-\otimes_\Lambda L$ and
choosing $\Lambda\in\mod \Lambda$ as the distinguished object. 
It follows easily from ampleness of $\O(1)$ on $\P^d$ that $s$ is ample.
Therefore we have
\[\mod \Lambda\simeq\qgr B(\Lambda,L),\]
which completes the proof.
\qed

\section{Examples}

	To get a better feel for  the tilting bundle $T$ from
	Section \ref{Tilting} let us compute
	$\End_\Lambda(T)$ in the case $d=1$. In this situation, 
	$[0,\c]=\{0,\c,a_i\x_i \;|\;1\leq a_i\leq p_i-1\}$. If $i\neq j$,
	$0< a_i<p_i$ and $0< a_j<p_j$ then \[\Hom_\Lambda(P(a_i\x_i),P(a_j\x_j))=
	H^0(\P^d, P(a_j\x_j-a_i\x_i))=H^0(\P^d,\O(-1))=0\] whilst
	\[\Hom_\Lambda(P(a_i\x_i),P(\left( a_i+1 \right)\x_i)=H^0(\P^d, \O)=1.\] Hence
	the endomorphism algebra is given by the following quiver with relations:
	
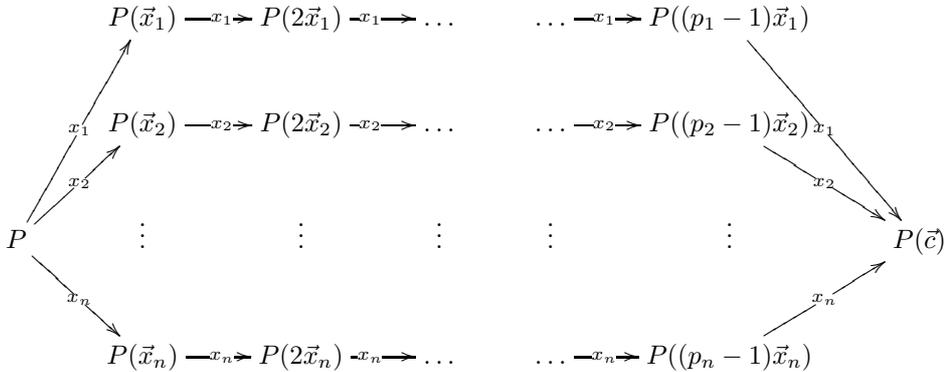
\[\xymatrix{
		&P(\x_1)\ar[r]|{x_1}&P(2\x_1)\ar[r]|{x_1}&\dots&\dots\ar[r]|(.3){x_1}&
		P((p_1-1)\x_1)\ar[ddr]|{x_1}\\
		&P(\x_2)\ar[r]|{x_2}&P(2\x_2)\ar[r]|{x_2}&\dots&\dots\ar[r]|(.3){x_2}
		&P((p_2-1)\x_2)\ar[dr]|{x_2}\\
		P\ar[uur]|{x_1}\ar[ur]|{x_2}\ar[dr]|{x_n}&\vdots&\vdots&\vdots&\vdots&\vdots&P(\c)\\
		&P(\x_n)\ar[r]|{x_n}&P(2\x_n)\ar[r]|{x_n}&\dots&\dots\ar[r]|(.3){x_n}&P((p_n-1)\x_n)\ar[ur]|{x_n}
	}	\] To see the relations, note firstly that
	by writing down our order we have implicitly chosen
	an $\eta_i\in H^0(\P^1,\O(L_i))$ for all $1\leq i\leq n$.
	If $n\geq 3$ then
	necessarily for all $i\geq 3$ we have $\eta_i=\ell_i(\eta_1,\eta_2)$
	for some functional $\ell_i$. The relations are thus
	\[x_{i}^{p_i}=\ell_i(x_1^{p_1},x_2^{p_2}),\quad{\rm for}\;
i\geq 3.\]

\begin{example}
	In this example we would like to show the relationship
	between a GL weighted $\P^1_{T_0:T_1}$ and the corresponding GL order on $\P^1$.
	Here, $\O=\O_{\P^1}$ and we choose 
$\lambda_1=(1:0),\lambda_2=(0:1), \lambda_3=(1:1)$ be three points on $\P^1$. 
Consider
\[\Lambda=
	\begin{bmatrix}
		\O&\O(-\lambda_1)\\
		\O&\O
\end{bmatrix}\otimes
\begin{bmatrix}
		\O&\O(-\lambda_2)\\
		\O&\O
\end{bmatrix}\otimes
\begin{bmatrix}
		\O&\O(-\lambda_3)\\
		\O&\O
	\end{bmatrix}
\]
In this case, $\L=\gen{\x_1,\x_2,\x_3,\c}/(2\x_1=2\x_2=2\x_3=\c)$ acts
on ${\mod \Lambda}$ where the action $\x_1$ is given by:
\begin{align*}
-\otimes_\Lambda I_1&=-\otimes_\Lambda
\left(\begin{bmatrix}
	\O&\O\\
	\O(\lambda_1)&\O
\end{bmatrix}\otimes
\begin{bmatrix}
	\O&\O(-\lambda_2)\\
	\O&\O
\end{bmatrix}\otimes
\begin{bmatrix}
	\O&\O(-\lambda_3)\\
	\O&\O
\end{bmatrix}\right)
\end{align*}
and similarly for the actions of $\x_2$ and $\x_3$. In this case
\[P=
	\begin{bmatrix}
		\O\\\O
	\end{bmatrix}\otimes
	\begin{bmatrix}
		\O\\\O
	\end{bmatrix}\otimes
	\begin{bmatrix}
		\O\\\O
	\end{bmatrix}
	\; {\rm and}\;P(\x_1)=
	\begin{bmatrix}
		\O\\\O(\lambda_1)
	\end{bmatrix}\otimes
	\begin{bmatrix}
		\O\\\O
	\end{bmatrix}\otimes
	\begin{bmatrix}
		\O\\\O
	\end{bmatrix}
\] and similarly for $P(\x_2)$ and $P(\x_3)$.
The tilting bundle is \[T=P\oplus P(\x_1)\oplus P(\x_2)\oplus P(\x_3)\oplus
P(1)\] with endomorphism algebra given by
\[\xymatrix{&P(\x_1)\ar|{x_1}[dr]\\
P\ar|{x_1}[ur]\ar|{x_2}[r]\ar|{x_3}[dr]&P(\x_2)\ar|{x_2}[r]&P(1)\\
&P(\x_3)\ar|{x_3}[ur]}\] We can always choose coordinates such that
$\O(-\lambda_i)\hookrightarrow\O$ is given by $T_{i-1}=0$ for $i=1,2$ and
$\O(-\lambda_3)\hookrightarrow\O$ is given by $T_1-T_2=0$
Thus the relation is $x_3^2=x_1^2-x_2^2$.
The corresponding $\L$-graded ring is then 
\begin{align*}
R&=k[T_0, T_1, X_1,X_2,X_3]/(X_1^2
-T_0, X_2^2-T_1, X_3^2-(T_0-T_1))\\
&\simeq k[X_1,X_2,X_3]/(X_3^2-(X_1^2-X_2^2)).
\end{align*}

\end{example}
\begin{example}
	Finally, we would like
	to present the simplest example possible on $\P^2_{T_0:T_1:T_2}$: one 
	with $4$ weights, all equaling $2$. Here 
	$\O=\O_{\P^2}$ and we let
	$L_i$ be the hyperplane given by $T_{i-1}=0$ for $1\leq i\leq 3$ and
let $L_4$ be given by $T_0+T_1+T_2=0$.
Consider \[\Lambda=
	\begin{bmatrix}
		\O&\O(-L_1)\\
		\O&\O
\end{bmatrix}\otimes
\begin{bmatrix}
		\O&\O(-L_2)\\
		\O&\O
\end{bmatrix}\otimes
\begin{bmatrix}
		\O&\O(-L_3)\\
		\O&\O
	\end{bmatrix}\otimes
\begin{bmatrix}
		\O&\O(-L_4)\\
		\O&\O
	\end{bmatrix}
\]
In this case $|[0,2\c]|=17$ and the endomorphism algebra is given by 
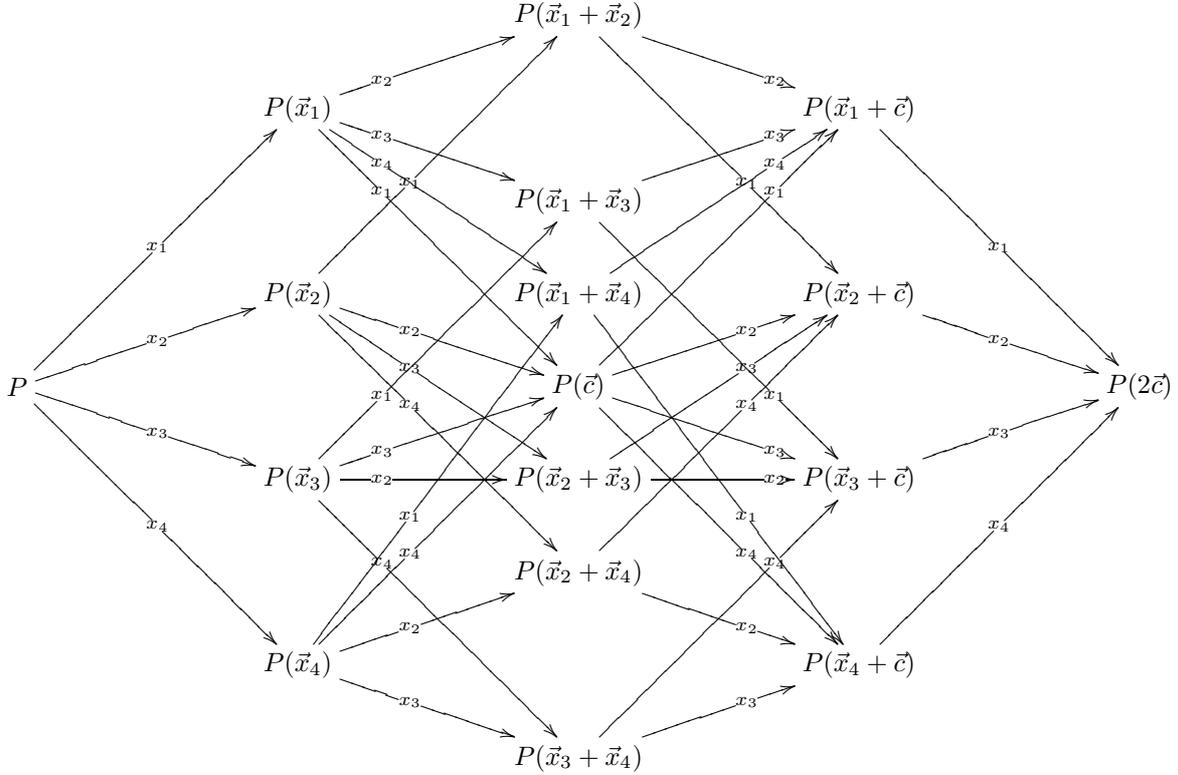
\[
\begin{xy} 0;<3.5pt,0pt>:<0pt,3.5pt>::
(-60,0) *+{P} ="0",
(-30,30) *+{P(\x_1)} ="1",
(-30,10) *+{P(\x_2)} ="2",
(-30,-10) *+{P(\x_3)} ="3",
(-30,-30) *+{P(\x_4)} ="4",
(0,40) *+{P(\x_1+\x_2)} ="12",
(0,20) *+{P(\x_1+\x_3)} ="13",
(0,10) *+{P(\x_1+\x_4)} ="14",
(0,-10) *+{P(\x_2+\x_3)} ="23",
(0,-20) *+{P(\x_2+\x_4)} ="24",
(0,-40) *+{P(\x_3+\x_4)} ="34",
(0,0) *+{P(\c)} ="c",
(30,30) *+{P(\x_1+\c)} ="c1",
(30,10) *+{P(\x_2+\c)} ="c2",
(30,-10) *+{P(\x_3+\c)} ="c3",
(30,-30) *+{P(\x_4+\c)} ="c4",
(60,0) *+{P(2\c)} ="2c",
"0", {\ar"1"|{x_1}},
"0", {\ar"2"|{x_2}},
"0", {\ar"3"|{x_3}},
"0", {\ar"4"|{x_4}},
"1", {\ar"c"|(.3){x_1}},
"1", {\ar"12"|(.3){x_2}},
"1", {\ar"13"|(.3){x_3}},
"1", {\ar"14"|(.3){x_4}},
"2", {\ar"12"|(.4){x_1}},
"2", {\ar"c"|(.4){x_2}},
"2", {\ar"23"|(.4){x_3}},
"2", {\ar"24"|(.4){x_4}},
"3", {\ar"13"|(.3){x_1}},
"3", {\ar"23"|(.3){x_2}},
"3", {\ar"c"|(.3){x_3}},
"3", {\ar"34"|(.3){x_4}},
"4", {\ar"14"|(.4){x_1}},
"4", {\ar"24"|(.4){x_2}},
"4", {\ar"34"|(.4){x_3}},
"4", {\ar"c"|(.4){x_4}},
"12", {\ar"c2"|(.6){x_1}},
"12", {\ar"c1"|(.7){x_2}},
"13", {\ar"c3"|(.7){x_1}},
"13", {\ar"c1"|(.7){x_3}},
"14", {\ar"c4"|(.6){x_1}},
"14", {\ar"c1"|(.7){x_4}},
"23", {\ar"c3"|(.7){x_2}},
"23", {\ar"c2"|(.6){x_3}},
"24", {\ar"c4"|(.6){x_2}},
"24", {\ar"c2"|(.6){x_4}},
"34", {\ar"c4"|(.6){x_3}},
"34", {\ar"c3"|(.7){x_4}},
"c", {\ar"c1"|(.7){x_1}},
"c", {\ar"c2"|(.6){x_2}},
"c", {\ar"c3"|(.7){x_3}},
"c", {\ar"c4"|(.6){x_4}},
"c1", {\ar"2c"|{x_1}},
"c2", {\ar"2c"|{x_2}},
"c3", {\ar"2c"|{x_3}},
"c4", {\ar"2c"|{x_4}},
\end{xy}
\]
with relations:
\begin{align*}
	x_ix_j&=x_jx_i\\
	x_4^2&=x^2_1+x_2^2+x_3^2
\end{align*}

\end{example}

\end{document}